\documentclass[10pt,a4paper]{amsart}
\usepackage{amssymb}

\theoremstyle{plain}
\newtheorem{theorem}{Theorem}[section]

\newtheorem{lemma}[theorem]{Lemma}

\theoremstyle{definition}

\theoremstyle{remark}

\numberwithin{equation}{section}
\numberwithin{theorem}{section}

\newcommand{\Z}{\mathbb{Z}}

\newcommand{\R}{\mathbb{R}}
\newcommand{\C}{\mathbb{C}}

\newcommand{\Ot}{\operatorname{O}}
\newcommand{\SO}{\operatorname{SO}}

\newcommand{\Iso}{\operatorname{Iso}}
\newcommand{\so}{\mathfrak{so}}

\newcommand{\ba}{\backslash}

\title[Representation equivalent Bieberbach groups]
{Representation equivalent Bieberbach groups and strongly isospectral flat manifolds}
\author{Emilio A.\ Lauret}
\address{Facultad de Matem\'atica Astronom\'ia y F\'isica (FaMAF), Universidad Nacional de C\'ordoba,  Medina Allende s/n, Ciudad Universitaria, X5000HUA, C\'ordoba, Argentina.}
\email{elauret@famaf.unc.edu.ar}
\thanks{Supported by CONICET and Secyt-UNC}
\date{April 15, 2013}

\begin{document}

\maketitle

\begin{abstract}
Let $\Gamma_1$ and $\Gamma_2$ be Bieberbach groups contained in the full isometry group $G$ of $\R^n$.
We prove that if the compact flat manifolds $\Gamma_1\ba\R^n$ and $\Gamma_2\ba\R^n$ are strongly isospectral then the Bieberbach groups $\Gamma_1$ and $\Gamma_2$ are representation equivalent, that is, the right regular representations $L^2(\Gamma_1\ba G)$ and $L^2(\Gamma_2\ba G)$ are unitarily equivalent.
\end{abstract}

\section{Introduction}

Let $X=G/K$ be a homogeneous Riemannian manifold where $G=\Iso(X)$ is the full isometry group of $X$ and where $K\subset G$ is a compact subgroup.
Let $\widehat G$ denote the unitary dual group of $G$.
Let $\Gamma_1$ and $\Gamma_2$ be discrete cocompact subgroups of $G$ acting on $X$ without fixed points.
The right regular representation $R_{\Gamma_i}$ on $L^2(\Gamma_i\ba G)$ splits as a direct sum
\begin{equation}\label{eq1:L^2}
L^2(\Gamma_i\ba G)=\sum_{(\pi,H_\pi)\in\widehat G} n_{\Gamma_i}(\pi)\,H_\pi
\end{equation}
where all the multiplicities $n_{\Gamma_i}(\pi)$ of $\pi$ in $L^2(\Gamma_i\ba G)$ are finite and only countably many are not zero.
The groups $\Gamma_1$ and $\Gamma_2$ are called \emph{representation equivalent} if the representations $R_{\Gamma_1}$ and $R_{\Gamma_2}$ are equivalent, that is, $n_{\Gamma_1}(\pi)=n_{\Gamma_2}(\pi)$ for every $\pi\in\widehat G$.

A generalized version of Sunada's theorem (see \cite[\S3]{Go} and the references therein) says that if the groups $\Gamma_1$ and $\Gamma_2$ are representation equivalent, then the compact manifolds $\Gamma_1\ba X$ and $\Gamma_2\ba X$ are \emph{strongly isospectral}, that is, for any natural bundle $E$ of $X$ and for any strongly elliptic natural operator $D$ acting on  sections of $E$, the associated operators $D_{\Gamma_1}$ and $D_{\Gamma_2}$ acting on sections of the bundles $\Gamma_1\ba E$ and $\Gamma_2\ba E$ have the same spectrum.

One may ask whether the converse holds.
Actually, little is known about this problem.
In \cite{Pe1}, H.~Pesce proved that the converse is true for spherical space forms ($X=S^n$) and for compact hyperbolic manifolds ($X=H^n$).
Our goal is to complete the picture within the class of spaces of constant curvature, by extending Pesce's result to the flat case ($X=\R^n$).

In his proof, Pesce only used the isospectrality with respect to certain natural operators.
More precisely, for $(\tau,V_\tau)\in\widehat K$, one associates the vector bundle $E_\tau=G \times_\tau V_\tau$ (see~\S\ref{sec:proof}).
Thus, the Laplace operator acting on sections of $E_\tau$ induces the operator $\Delta_{\tau,\Gamma_1}$ and $\Delta_{\tau,\Gamma_2}$ acting on sections of $\Gamma_1\ba E_\tau$ and $\Gamma_2\ba E_\tau$ respectively.
We will say that the manifolds $\Gamma_1\ba X$ and $\Gamma_2\ba X$ are \emph{$\tau$-isospectral} if these operators have the same spectrum.

In our case, the isometry group of $X=\R^n$ is $G=\Ot(n)\ltimes\R^n$ and a discrete cocompact subgroup of $G$ acting without fixed points is usually called a \emph{Bieberbach group}.
We state the theorem in a way analogous to that in \cite{Pe1}.

\begin{theorem}\label{thm1:main}
Let $\Gamma_1$ and $\Gamma_2$ be Bieberbach groups contained in the full isometry group $G$ of $\R^n$.
The following assertions are equivalent:
\begin{enumerate}
  \item\label{item1:L^2}
  $\Gamma_1$ and $\Gamma_2$ are representation equivalent.
  \item\label{item1:strongly-iso}
  $\Gamma_1\ba\R^n$ and $\Gamma_2\ba\R^n$ are strongly isospectral.
  \item\label{item1:tau-iso}
  $\Gamma_1\ba\R^n$ and $\Gamma_2\ba\R^n$ are $\tau$-isospectral for every $\tau\in\widehat K$.
\end{enumerate}
\end{theorem}

As we mentioned above, (\ref{item1:L^2})$\Rightarrow$(\ref{item1:strongly-iso}) is well known and (\ref{item1:strongly-iso})$\Rightarrow$(\ref{item1:tau-iso}) holds trivially.
It suffices to show (\ref{item1:tau-iso})$\Rightarrow$(\ref{item1:L^2}) (see page \pageref{proof:main-thm}).
The techniques used here are similar to those in \cite{LMR}, where the $p$-spectrum of any constant curvature space form is determined in terms of the multiplicities in \eqref{eq1:L^2}.

The author is greatly indebted to Juan Pablo Rossetti for suggesting the problem and for many stimulating conversations, and also to Roberto Miatello for his active interest in the publication of this paper.
The author also wishes to express his thanks to the referee for several useful comments and corrections.

\section{Preliminaries}

\subsection{Irreducible representations of orthogonal groups}\label{subsec:dual-ortho}
In this subsection we describe the unitary dual group of the orthogonal group $\Ot(n)$.
Furthermore, we recall the branching laws to $\Ot(n-1)$.

We write $n=2m$ if $n$ is even or $n=2m+1$ if $n$ is odd.
We fix the Cartan subalgebra of $\so(n,\C)$ as
$
\mathfrak h=\left\{H=\sum_{j=1}^m i h_j(E_{2j-1,2j} - E_{2j,2j-1}): h_j\in\C\right\}.
$
For $H\in\mathfrak h$, set $\varepsilon_j(H)=h_j$ for $1\leq j\leq m$.
The highest weight theorem gives a one-one correspondence between the irreducible representations of $\SO(n)$ and the elements in $\mathcal P(\SO(n))$, that is, the dominant analytically integral linear functionals on $\mathfrak h$.
The correspondence being that $\Lambda$ is the highest weight of the representation.
We have
\begin{align*}\label{eq2:P(SO(n))}
\mathcal{P}(\SO(2m+1))
    &= \left\{\textstyle\sum\limits_{i=1}^m  a_i\varepsilon_i:
        a_i\in\Z\;\forall i,\; a_1\geq a_2\geq\dots\geq a_m\geq0 \right\},\\
\mathcal{P}(\SO(2m))
    &= \left\{\textstyle\sum\limits_{i=1}^m  a_i\varepsilon_i:
        a_i\in\Z\;\forall i,\; a_1\geq\dots\geq a_{m-1}\geq |a_m| \right\}.
\end{align*}
For $\Lambda\in\mathcal P(\SO(n))$, let $(\tau_\Lambda,V_\Lambda)$ denote the irreducible representation of $\SO(n)$ with highest weight $\Lambda$.

We now describe the irreducible representations of the full orthogonal group $\Ot(n)$ in terms of the irreducible representations of $\SO(n)$.
We let
\begin{equation}\label{eq2:g_0}
g_0=
\begin{cases}
-\mathrm{Id}_{n} &\quad
    \text{if $n$ is odd,}\\[1mm]
\left[\begin{smallmatrix}
\mathrm{Id}_{n-1}&\\ &-1
\end{smallmatrix}\right]&\quad
    \text{if $n$ is even,}
\end{cases}
\qquad\text{thus }\Ot(n)=\SO(n)\cup g_0\SO(n).
\end{equation}
It suffices to define the representations of $\Ot(n)$ on each connected component.

We first consider $n$ odd.
For $\Lambda\in\mathcal P(\SO(2m+1))$ and $\delta=\pm1$ we define a representation $(\tau_{\Lambda,\delta},V_\Lambda)$ of ${\Ot}(2m+1)$ by setting
\begin{equation*}
\tau_{\Lambda,\delta}(g)(v)=
    \begin{cases}
      \tau_\Lambda(g)(v)     &\;\text{if }g\in\SO(2m+1),\\
      \delta\,\tau_\Lambda( g_0g)( v) &\;\text{if }g\in g_0\SO(2m+1).
    \end{cases}
\end{equation*}
Clearly $\tau_{\Lambda,\delta}|_{\SO(2m+1)}\cong \tau_{\Lambda}$.
These representations are irreducible and every irreducible representation can be constructed in this way, thus
\begin{equation}\label{eq2:hatO(2m+1)}
\widehat{\Ot(2m+1)} = \left\{\tau_{\Lambda,\delta}: \Lambda\in\mathcal P(\SO(2m+1)),\; \delta\in\{\pm1\}\right\}.
\end{equation}

The even case is more complicated (see \cite[Subsection 2.2]{LMR} for more details).
Set $\overline\Lambda = \sum_{i=1}^{m-1} a_i\,\varepsilon_i-a_m\varepsilon_m$ if $\Lambda=\sum_{i=1}^m a_i\,\varepsilon_i\in\mathcal P(\SO(2m))$.
For $\Lambda\in \mathcal P(\SO(2m))$ satisfying $\Lambda=\overline\Lambda$ (i.e.\ $a_m=0$) and $\delta\in\{\pm1\}$, one associates $\tau_{\Lambda,\delta}\in\widehat{\Ot(2m)}$ on the vector space $V_\Lambda$.
Again we have that $\tau_{\Lambda,\delta}|_{\SO(2m)}\cong \tau_{\Lambda}$.
The parameter $\delta$ depends on certain intertwining operator $T_\Lambda$ (see \cite[p.~372]{Pe1} and \cite[eq.~(2.7)]{LMR}).
In the case $\Lambda\neq\overline\Lambda$ (i.e.\ $a_m\neq0$), there is a single representation $\tau_{\Lambda,0}\in\widehat{\Ot(2m)}$ defined on the vector space $V_{\Lambda}\oplus V_{\overline{\Lambda}}$, which satisfies $\pi_{\Lambda,0}|_{\SO(2m)}\cong \pi_{\Lambda}\oplus\pi_{\overline\Lambda}$.
Hence
\begin{align}\label{eq2:hatO(2m)}
\widehat{\Ot(2m)}
    & = \left\{\tau_{\Lambda,\delta}: \Lambda=\textstyle\sum\limits_{i=1}^m a_i\,\varepsilon_i \in\mathcal P(\SO(2m)),\; a_m=0,\, \delta\in\{\pm1\}\right\}. \\
    &\qquad\bigcup\quad \left\{\tau_{\Lambda,0}: \Lambda=\textstyle\sum\limits_{i=1}^m a_i\,\varepsilon_i \in\mathcal P(\SO(2m)),\; a_m>0\right\}. \notag
\end{align}
We shall use the notation $\tau_{\Lambda,\delta}$ in both cases, with the understanding that either $\delta=\pm1$ or $\delta=0$ according to $a_m=0$ or $a_m\neq0$ respectively.

One check that
\begin{equation*}
\tau_{\Lambda,\delta} \simeq
    \begin{cases}
      \tau_{\Lambda,-\delta}\otimes\det , \\
      \tau_{\Lambda,\delta}^*,
    \end{cases}
    \qquad\text{for any $\tau_{\Lambda,\delta}\in\widehat{\Ot(n)}$.}
\end{equation*}

\smallskip

We conclude this subsection by stating the branching laws from $\Ot(n)$ to $\Ot(n-1)$.

\begin{theorem}\label{thm2:branching-even}
Let $\tau_{\Lambda,\delta}\in\widehat{\Ot(2m)}$ with $\Lambda=\sum_{i=1}^m a_i\,\varepsilon_i$ and $\delta\in\{0,\pm1\}$.
If $a_m>0$ (resp.\ $a_m=0$), then
$
\tau_{\Lambda,\delta}|_{\Ot(2m-1)} = \sum \;\sigma_{\mu,\kappa},
$
where the sum is over all $\mu=\sum_{i=1}^{m-1} b_i\varepsilon_i$ such that
\begin{align*}
& a_1\geq b_1\geq a_2\geq b_2\geq\dots \geq a_{m-1}\geq b_{m-1}\geq a_m
\end{align*}
and any $\kappa\in\{\pm1\}$ (resp.\ a single value of $\kappa\in\{\pm1\}$).
\end{theorem}

\begin{theorem}\label{thm2:branching-odd}
Let $\tau_{\Lambda,\delta}\in\widehat{\Ot(2m+1)}$, where $\Lambda=\sum_{i=1}^m a_i\,\varepsilon_i$ and $\delta\in\{\pm1\}$.
Then
$
\tau_{\Lambda,\delta}|_{\Ot(2m)} = \sum \;\sigma_{\mu,\kappa},
$
where the sum is over all $\mu=\sum_{i=1}^{m} b_i\varepsilon_i$ such that
\begin{align*}
& a_1\geq b_1\geq a_2\geq b_2\geq\dots \geq a_{m-1}\geq b_{m-1}\geq a_m\geq b_m\geq0
\end{align*}
and, a single value of $\kappa\in\{\pm1\}$ if $b_m=0$ or $\kappa=0$ if $b_m>0$.
\end{theorem}

Note that in both cases the branching is multiplicity free, that is, the multiplicity $[\sigma_{\mu,\kappa}:\tau_{\Lambda,\delta}]:=\dim_{\Ot(n-1)} (W_\sigma,V_\tau)$ is always equal to $0$ or $1$.

\subsection{Unitary dual of $\Iso(\R^n)$}
Here we describe the unitary irreducible representations of $\Ot(n)\ltimes\R^n\simeq\Iso(\R^n)$ (see \cite[\S4]{LMR} for more details).
We write any element $g\in \Ot(n)\ltimes\R^n$ as $g=B L_b$, where $B\in \Ot(n)$ is called the rotational part, and $L_b$ denotes translation by $b\in\R^n$.
From now on, we fix the following notation:
\begin{align}
G &= \Ot(n)\ltimes \R^n, \notag\\
K &= \Ot(n), \label{eq2:G-K-M}\\
M &= \left\{\left(\begin{smallmatrix}B&\\&\det(B)\end{smallmatrix}\right): B\in \Ot(n-1)\right\}.\notag
\end{align}

An element $(\tau,W_\tau)\in\widehat K$ induces a representation $\widetilde\tau$ of $G$ on $W_\tau$ given by
$$
\widetilde\tau(BL_b) (w) = \tau(B) (w).
$$
In other words, $\widetilde\tau=\tau\otimes\mathrm{Id_{W_\tau}}$.
Clearly, $\widetilde\tau$ is finite dimensional, unitary and irreducible.

We identify $\widehat{\R}^n$ with $\R^n$ via the correspondence $\alpha \mapsto \xi_\alpha(\,\cdot\,) = e^{2\pi i \langle \alpha,\, \cdot\, \rangle}$ for $\alpha \in \R^n$.
Under the notation given by \eqref{eq2:G-K-M}, given $r>0$ and $(\sigma,V_\sigma)\in\widehat M$, we consider the induced representation of $G$ given by
\begin{equation*}
(\pi_{\sigma,r}, H_{\sigma,r}):=\operatorname{Ind}_{M\ltimes\R^n}^{K\ltimes\R^n} (\sigma\otimes \xi_{r e_n}).
\end{equation*}
It is well-known that $\pi_{\sigma,r}$ is unitary and irreducible.

Finally, a full set of representatives of $\widehat G$ is given by
\begin{eqnarray}\label{eq2:dual_E(n)}
\widehat G =
    \left\{\widetilde\tau:\tau\in\widehat{K}\right\} \;\cup\;
    \left\{\pi_{\sigma,r}:\sigma\in\widehat{M},\; r>0\right\}.
\end{eqnarray}

\section{Main Theorem}\label{sec:proof}
In this section we prove Theorem~\ref{thm1:main}.
We still use the notation in \eqref{eq2:G-K-M} for the groups $G$, $K$ and $M$, and $\{e_1,\dots,e_n\}$ denotes the canonical basis of $\R^n$.

We recall some notions on homogeneous vector bundles (see \cite[\S5.2]{Wa} or \cite[Subsection 2.1]{LMR}) on $\R^n$ and compact flat manifolds.
Let $(\tau,V_\tau)$ be a unitary representation of $K$ of finite dimension.
The homogeneous vector bundle $E_\tau= G \times_\tau V_\tau$ of $X$ is constructed as $G\times V_\tau/\sim$ where $(x,v)\sim (xk,\tau (k^{-1})v)$ for every $k\in K$.
The group $G$ acts on $E_\tau$ by $g\cdot [x,v]=[gx,v]$, where $[x,v]$ denotes the class of equivalence of $(x,v)$.
The space of smooth sections $\Gamma^\infty(E_\tau)$ of $E_\tau$ is in correspondence with the set $C^\infty(G;\tau)$, the smooth functions $f:G\to V_\tau$ such that $f(xk)=\tau(k^{-1})f(x)$ for every $k\in K$ and $x\in G$.
The element
$$
C:=e_1^2+\dots+e_n^2\in U(\mathfrak g)
$$
defines a differential operator $\Delta_\tau$ on $\Gamma^\infty(E_\tau)$.
For example, if $\tau=\mathbf{1}$, the trivial representation of $K$, then
$$
\Delta_{\mathbf{1}} = \frac{\partial^2}{\partial x_1^2} + \dots + \frac{\partial^2}{\partial x_n^2},
$$
is just the Laplace operator on $\R^n$.
Furthermore, the element $C$ commutes with every irreducible representation of $G$ contained in $C^\infty(G;\tau)$, thus by Schur's lemma,  $C$ acts by an scalar $\lambda(\pi,C)$ on $V_\pi$ for every $\pi\in\widehat G$ such that $V_\pi\subset C^\infty(G;\tau)$.

The quotient $\Gamma\ba E_\tau$ is a homogeneous vector bundle over the compact flat manifold  $\Gamma\ba \R^n$, and again, the element $C$ defines a differential operator $\Delta_{\tau,\Gamma}$ acting on the sections of $\Gamma\ba E_\tau$.
Given $\Gamma_1$ and $\Gamma_2$, two Bieberbach groups in $G$, the spaces $\Gamma_1\backslash \R^n$ and $\Gamma_2\backslash \R^n$ are said to be \emph{$\tau$-isospectral} if $\Delta_{\tau,\Gamma_1}$ and $\Delta_{\tau,\Gamma_2}$ have the same spectrum.

We now determine the $\tau$-spectrum for any $\tau\in\widehat K$ as in \cite{LMR}.
Recall that $n_\Gamma(\pi)$ ($\pi\in\widehat G$) denotes the multiplicity of $\pi$ in $L^2(\Gamma\ba G)$ as we introduce in \eqref{eq1:L^2}.
We shall use the following notation:
\begin{align*}
&\left\{\begin{array}{r@{\;}l}
\widehat G(0)&=\{\widetilde\tau:\tau\in\widehat K\},\\[1mm]
\widehat G(\sigma)&= \{\pi_{\sigma,r}:r>0\}\quad (\sigma\in\widehat M),
\end{array}\right.
&\text{thus}\qquad
\widehat G=\widehat G(0)\cup\bigcup_{\sigma\in\widehat M} \widehat G(\sigma).
\end{align*}

\begin{theorem}
Let $\tau\in\widehat K$ and $\lambda\in\R$.
The multiplicity $d_\lambda(\tau,\Gamma)$ of $\lambda$ in the spectrum of $\Delta_{\tau,\Gamma}$ is given by
\begin{equation}\label{eq3:tau-spec}
d_\lambda(\tau,\Gamma) =
\begin{cases}
0&\quad\text{if }\,\lambda<0,\\[3mm]
n_\Gamma(\widetilde\tau)  &\quad\text{if }\,\lambda=0,\\[3mm]
\displaystyle\sum_{\sigma\in\widehat M:\,[\sigma:\tau|_M]>0} n_\Gamma(\pi_{\sigma,\sqrt\lambda/2\pi})
    &\quad\text{if }\,\lambda>0.
\end{cases}
\end{equation}
\end{theorem}

\begin{proof}
For any locally symmetric space we have that (see \cite[Prop.~2.4]{LMR})
\begin{equation}\label{eq3:d_lambda1}
d_{\lambda}(\tau,\Gamma)=
    \sum_{\pi\in \widehat G\atop \lambda(C,\pi)=\lambda}
    n_\Gamma(\pi)\, [\tau^*:\pi|_K].
\end{equation}
Note that the sum is already over the elements in $\widehat G_{\tau^*}=\widehat G_{\tau}:=\{\pi\in\widehat G:\, [\tau:\pi|_K]>0\}$, since $\tau^*\simeq\tau$.
Since $\widetilde \tau_0|_K=\tau_0$ for any $\tau_0\in\widehat K$, it follows that $\widehat G_{\tau}\cap\widehat G(0) = \{\widetilde{\tau}\}$.
On the other hand, $[\tau:\pi_{\sigma,r}|_K] = [\tau: \operatorname{Ind}_{M}^K(\sigma)]= [\sigma:\tau|_{M}]$ by Frobenius reciprocity.
Then
\begin{equation*}
\widehat G_{\tau} =
    \left\{\widetilde \tau\right\} \cup
    \bigcup_{\sigma\in\widehat M \atop [\sigma:\tau|_M]>0} \widehat G(\sigma).
\end{equation*}
The branching rules given in Theorems~\ref{thm2:branching-even} and \ref{thm2:branching-odd} give a complete description of $\widehat G_\tau$ in terms of highest weights.
Moreover, they also ensure that $[\tau:\pi|_K]=1$ for every $\pi\in\widehat G_{\tau}$.

Finally, by Schur's lemma, the element $C$ acts by a scalar $\lambda(C,\pi)$ on each $H_\pi$.
We have (see \cite[Lem.~4.2]{LMR})
\begin{equation}\label{eq3:Casimir_on_pi}
\lambda(C,\pi) =
\begin{cases}
  0&    \quad\text{for }\pi\in\widehat G(0),\\
  -4\pi^2 r^2 & \quad\text{for }\pi=\pi_{\sigma,r}\in\widehat G(\sigma),
\end{cases}
\end{equation}
which concludes the proof.
\end{proof}

We are now in a position to prove the main theorem.

\begin{proof}[Proof of Theorem~\ref{thm1:main}]\label{proof:main-thm}
We have to prove that
\begin{equation}\label{eq3:star}
n_{\Gamma_1}(\pi)=n_{\Gamma_2}(\pi)
\end{equation}
for every $\pi\in\widehat G$, by assuming that $d_\lambda(\tau,\Gamma_1)=d_\lambda(\tau,\Gamma_2)$ for every $\lambda\in\R$ and every $\tau\in\widehat K$.
From \eqref{eq3:tau-spec} for the eigenvalue $\lambda=0$, it follows that \eqref{eq3:star} holds for every $\pi\in\widehat G(0)$.

It remains to prove that, for any $\sigma\in\widehat M$, \eqref{eq3:star} holds for every $\pi\in\widehat G(\sigma)$.
We shall do this by the repeated application of the following lemmas.
We write $n=2m$ if $n$ is even and $n=2m+1$ if $n$ is odd.
For $\mu_1=\sum_{i=1}^m b_i\varepsilon_i$ and $\mu_2=\sum_{i=1}^m c_i\varepsilon_i$ in $\mathcal P(\SO(n))$ with $b_m,c_m\geq0$, we write $\mu_1<\mu _2$ if $c_1-b_1\geq c_2-b_2\geq \dots\geq c_m-b_m\geq 0$, and set $\ell(\mu)=p$ if $b_p\neq0$ and $b_i=0$ for all $i>p$.

\begin{lemma}\label{lem3:lemma}
Let $\Gamma_1$ and $\Gamma_2$ be Bieberbach groups in $G$ and let $\mu_0\in\mathcal P(\SO(n-1))$.
If $\Gamma_1\ba\R^n$ and $\Gamma_2\ba\R^n$ are $\tau_{\mu_0,\delta}$-isospectral and $n_{\Gamma_1}(\pi)=n_{\Gamma_2}(\pi)$ for every $\pi\in\bigcup_{\mu<\mu_0}\widehat G(\sigma_{\mu,\kappa})$, then $n_{\Gamma_1}(\pi)=n_{\Gamma_2}(\pi)$ for every $\pi\in\widehat G(\sigma_{\mu_0,\kappa})$.
\end{lemma}
\begin{proof}
Write $\mu_0=\sum_{i=1}^m b_i\varepsilon_i \in \mathcal P(\SO(n-1))$, with the convention that $b_m=0$ if $n$ is even, since $n-1=2m-1$.
Let $\Lambda=\sum_{i=1}^m b_i\varepsilon_i\in\mathcal P(\SO(n))$.
Theorems~\ref{thm2:branching-even} and \ref{thm2:branching-odd} ensure that $[\sigma_{\mu,\kappa}:\tau_{\Lambda,\delta}|_M]>0$ if and only if $\mu=\sum_{i=1}^m c_i\varepsilon_i$ satisfies
$$
b_1\geq c_1\geq b_2\geq c_2\geq \dots \geq b_m\geq c_m\geq 0,
$$
and for a single value of $\kappa\in\{0,\pm1\}$.
Now, by \eqref{eq3:tau-spec} we have that
$$
d_{4\pi^2r^2}(\tau_{\Lambda,\pm\delta},\Gamma_i) = n_{\Gamma_i}(\pi_{\sigma_{\mu_0,\pm\kappa_0},r}) + \sum_{[\sigma_{\mu,\kappa}:\tau_{\Lambda,\delta}]>0\atop \mu\neq\mu_0} n_{\Gamma_i}(\pi_{\sigma_{\mu,\pm\kappa_\mu},r})
$$
for every $r>0$.
It is clear that, if $[\sigma_{\mu,\kappa}:\tau_{\Lambda,\delta}|_M]>0$ then $\mu=\mu_0$ or $\mu<\mu_0$.
Hence $n_{\Gamma_1}(\pi_{\sigma_{\mu_0,\pm\kappa_0},r}) = n_{\Gamma_2}(\pi_{\sigma_{\mu_0,\pm\kappa_0},r})$ for every $r>0$ since we are assuming that $d_{4\pi^2r^2}(\tau_{\Lambda,\pm\delta},\Gamma_1)= d_{4\pi^2r^2}(\tau_{\Lambda,\pm\delta},\Gamma_2)$ ($\tau_{\mu_0,\kappa}$-isospectrality) and $n_{\Gamma_1}(\pi_{\sigma_{\mu,\pm\kappa_\mu},r}) = n_{\Gamma_2}(\pi_{\sigma_{\mu,\pm\kappa_\mu},r})$ for every $\mu<\mu_0$ and every $r>0$.
\end{proof}

For example, by applying Lemma~\ref{lem3:lemma} to $\mu_0=0$, we obtain that \eqref{eq3:star} holds for every $\pi\in\widehat G(\sigma_{0,\kappa})$, which is the same result of \cite[Prop.~3.2~(c)]{Pe2}.

\begin{lemma}\label{lem3:lemma2}
Let $\Gamma_1$ and $\Gamma_2$ be Bieberbach groups in $G$ and let $1\leq p<m$.
If $\Gamma_1\ba\R^n$ and $\Gamma_2\ba\R^n$ are $\tau_{\mu,\delta}$-isospectral for every $\mu\in\mathcal P(\SO(n))$ such that $\ell(\mu)=p+1$ and $n_{\Gamma_1}(\pi)=n_{\Gamma_2}(\pi)$ for every $\pi\in\widehat G(\sigma_{\mu,\kappa})$ such that $\ell(\mu)\leq p$, then $n_{\Gamma_1}(\pi)=n_{\Gamma_2}(\pi)$ for every $\pi\in\widehat G(\sigma_{\mu,\kappa})$ such that $\ell(\mu)=p+1$.
\end{lemma}

\begin{proof}
We begin by considering the first case $\mu_0=\varepsilon_1+\dots+\varepsilon_{p+1}\in\mathcal P(\SO(n))$.
Every $\mu\in\mathcal P(\SO(n))$ such that $\mu <\mu_0$ satisfies $\ell(\mu)\leq p$, thus $n_{\Gamma_1}(\pi)=n_{\Gamma_2}(\pi)$ for every $\pi\in\cup_{\mu<\mu_0}\widehat G(\sigma_{\mu,\kappa})$.
Hence $n_{\Gamma_1}(\pi)=n_{\Gamma_2}(\pi)$ for every $\pi\in\widehat G(\sigma_{\mu_0,\kappa})$ by Lemma~\ref{lem3:lemma}.

We continue in this fashion obtaining that $n_{\Gamma_1}(\pi)=n_{\Gamma_2}(\pi)$ for every $\pi\in\widehat G(\sigma_{\mu_0,\kappa})$ for any $\mu_0=\sum_{i=1}^{p} b_i \varepsilon_i + \varepsilon_{p+1}\in\mathcal P(\SO(n))$ since the ordering $<$ is complete.
We now proceed by induction on $b_{p+1}$, and the proof is complete.
\end{proof}

By proceeding by induction on $p$ with repeated applications of Lemma~\ref{lem3:lemma2}, we have that \eqref{eq3:star} holds for every $\pi\in\widehat G(\sigma_{\mu,\kappa})$ for any $\mu\in\mathcal P(\SO(n))$.
This completes the proof.
\end{proof}

\bibliographystyle{plain}

\end{document}